\pdfoutput=1 

\documentclass[11pt,a4paper,leqno]{article}

\usepackage{a4wide}

\usepackage[p,osf]{scholax}
\usepackage{amsmath,amsthm}
\usepackage[scaled=1.075,ncf,vvarbb]{newtxmath}

\usepackage{microtype}

\usepackage[only,llbracket,rrbracket]{stmaryrd}

\usepackage{graphicx}

\usepackage{enumitem}
\setlist[enumerate,1]{label=(\alph*)}
\setlist[enumerate,2]{label=(\roman*)}

\usepackage[T1]{fontenc}
\usepackage[utf8x]{inputenc}
\usepackage[usenames,dvipsnames]{xcolor}

\usepackage{url}
\usepackage[hidelinks,colorlinks=true,linkcolor=blue,citecolor=darkgray,linktocpage=false]{hyperref}

\usepackage{mathtools}
\mathtoolsset{showonlyrefs}

\usepackage{dsfont}

\newtheoremstyle{citing}
  {3pt}
  {3pt}
  {\itshape}
  {}
  {\bfseries}
  {.}
  {.5em}
  {\thmnote{#3}}
\theoremstyle{citing}

\theoremstyle{definition}

\theoremstyle{remark}

\swapnumbers
\theoremstyle{plain}

\newtheorem{theorem}{Theorem}[section]

\newtheorem{lemma}[theorem]{Lemma}

\newtheorem{corollary}[theorem]{Corollary}

\theoremstyle{remark}
\newtheorem{remark}[theorem]{Remark}
\newtheorem*{remark*}{Remark}

\theoremstyle{definition}
\newtheorem{definition}[theorem]{Definition}
\newtheorem*{definition*}{Definition}
\newtheorem{miniremark}[theorem]{}

\newtheorem*{question*}{Question}

\theoremstyle{definition}
\swapnumbers

\usepackage{chngcntr}

\counterwithin*{equation}{theorem}
\counterwithin*{enumi}{theorem}

\newcommand{\J}{\mathbf{J}}

\newcommand{\ccspace}[1]{\mathscr{K}(#1)}

\newcommand{\dspace}[2]{\mathscr{D}(#1,#2)}

\newcommand{\project}[1]{{#1}_\natural}

\newcommand{\eqproject}[1]{(#1)_\natural}

\newcommand{\grass}[2]{\mathbf{G}(#1,#2)}

\newcommand{\oball}[2]{\mathbf{U}(#1,#2)}

\newcommand{\cball}[2]{\mathbf{B}(#1,#2)}

\ifdefined\textint
    \renewcommand{\textint}[2]{{\textstyle\int_{#1}^{#2}}}
\else    
    \newcommand{\textint}[2]{{\textstyle\int_{#1}^{#2}}}
\fi

\ifdefined\textfint
    \renewcommand{\textfint}[2]{{\textstyle\fint_{#1}^{#2}}}
\else    
    \newcommand{\textfint}[2]{{\textstyle\fint_{#1}^{#2}}}
\fi

\ifdefined\textsum
    \renewcommand{\textsum}[2]{{\textstyle\sum_{#1}^{#2}}}
\else  
    \newcommand{\textsum}[2]{{\textstyle\sum_{#1}^{#2}}}
\fi

\ifdefined\textint
    
\else    
    
\fi

\newcommand{\natp}{\mathscr{P}}

\newcommand{\R}{\mathbf{R}}

\newcommand{\LM}{\mathscr{L}}

\newcommand{\Lp}[1]{\mathbf{L}_{#1}}

\DeclareMathOperator{\dist}{dist}

\newcommand{\HM}{\mathscr{H}}

\newcommand{\restrict}{ \mathop{ \rule[0pt]{.5pt}{6pt} \rule[0pt]{4pt}{0.5pt} }\nolimits }

\newcommand{\ud}{\ensuremath{\,\mathrm{d}}}

\newcommand{\uD}{\ensuremath{\mathrm{D}}}
\DeclareMathOperator{\trace}{trace}

\DeclareMathOperator{\Hom}{Hom}

\DeclareMathOperator{\Bdry}{Bdry}

\newcommand{\cnt}[1]{\mathscr{C}^{#1}}

\DeclareMathOperator{\dmn}{dmn}

\DeclareMathOperator{\grad}{grad}

\DeclareMathOperator{\spt}{spt}

\DeclareMathOperator{\Lip}{Lip}

\DeclareMathOperator{\card}{card}

\newcommand{\without}{\!\sim\!}

\DeclareMathOperator{\im}{im}

\newcommand{\vwedge}[1]{{\textstyle \boldsymbol{\bigwedge}_{#1}}}

\title{Regularity for $H$-systems of bounded distortion}
\author{S{\l}awomir Kolasi{\'n}ski}
\date{\today}

\begin{document}

\maketitle

\begin{abstract}
    We consider weak solutions $u \in W^{1,n}(U,\R^{n+1})$ of the $H$-system
    $- \operatorname{div} ( |\uD u|^{n-2} \uD u ) = (H \circ u) \, \J u$, where $n \ge 2$, $U$ is
    the open unit ball in~$\R^n$, and $H$ is bounded.  We show that $u$ is locally bounded if its
    distortion $|\uD u|^n / |\J u|$ is bounded on the set where $|u|$ is large. By a result of
    Fusco, Kristensen, Leone, and Verde, it follows that $u$ is locally H{\"o}lder continuous if $H$
    is continuous, and continuously differentiable with locally H{\"o}lder continuous derivative if
    $H$ is locally Lipschitz. The proof pushes forward $|\uD u|^{n-2} \uD u \circ \uD u^*$ by~$u$ to
    a linear functional on continuous compactly supported maps of $\R^{n+1}$ into
    $\Hom(\R^{n+1},\R^{n+1})$ , whose values on derivatives are controlled by the equation. When the
    distortion is bounded, this functional satisfies the hypotheses of an inequality of
    Michael--Simon type for measures satisfying a first-order partial differential equation, due to
    De~Philippis, Gennaioli, Pigati, and Rindler. The inequality yields lower bounds for the mass
    ratios of the push-forward of $|\uD u|^n \LM^n$ by~$u$, and these imply boundedness.
\end{abstract}

\section{Introduction}
\label{sec:intro}

Let $n \ge 2$ be an integer, $U = \oball 01 \subseteq \R^n$, and $H : \R^{n+1} \to \R$ be
bounded. We consider weak solutions $u \in W^{1,n}(U,\R^{n+1})$ of the $H$-system
\begin{displaymath}
    - \operatorname{div} \bigl( |\uD u|^{n-2} \uD u \bigr) = (H \circ u) \, \J u \,,
\end{displaymath}
where $\J u$ is the vector product of the partial derivatives of~$u$; see~\ref{def:jacobian}.

For $n = 2$ the regularity theory is well understood. Wente~\cite[Theorem~5.3]{Wen69} proved that weak solutions
are continuous when $H$ is constant and the solution agrees on the boundary of the disc with a
harmonic map of class $W^{1,2}$ which is continuous on the closed
disc. Gr{\"u}ter~\cite[3.9(i)]{Gru81} proved continuity of conformal solutions for bounded~$H$, and
Bethuel~\cite[Theorem~1]{Bet92} of all weak solutions for bounded
Lipschitz~$H$. Rivi{\`e}re~\cite[Theorem~I.1, (I.15)]{Riv07} proved continuity for every bounded~$H$,
confirming a conjecture of Heinz. He wrote the system as an elliptic system with an antisymmetric
potential and proved continuity for all such systems.

For $n \ge 3$, Duzaar and Fuchs~\cite[Theorem~2.2]{DF91} proved that bounded weak solutions are continuous
when $H$ is bounded and Lipschitz. Mou and Yang~\cite[Theorem~16]{MY96} proved continuity of
conformal solutions for bounded~$H$, extending Gr{\"u}ter's result. For constant $H$ they proved
that all weak solutions are of class $C^{1,\alpha}$~\cite[Theorem~3.6]{MY96b}. Wang~\cite[Theorem~1]{Wan99} proved that all weak solutions are
of class $C^{1,\alpha}$ when
$\sup |H| + \sup \{ (1 + |y|) \, \|\uD H(y)\| : y \in \R^{n+1} \} < \infty$. Duzaar and
Grotowski~\cite[Theorems~4.4 and~6.1]{DG00} obtained solutions of the Dirichlet problem by minimisation and
proved that they are H{\"o}lder continuous. For bounded globally Lipschitz $H$, Martino and
Schikorra~\cite[Corollary~1.8]{MaSc24} proved continuity under the assumption
$\uD u \in \Lp{(n,n/(n-1))}(U,\Hom(\R^n,\R^{n+1}))$, where $\Lp{(p,q)}$ denotes the Lorentz
space. Mi{\'s}kiewicz, Petraszczuk, and Strzelecki~\cite[Theorem~1.1]{MPS23} proved local H{\"o}lder
continuity for bounded Lipschitz $H$ when $n$ is even and $u \in W^{n/2,2}(U,\R^{n+1})$. The latter
result was announced for $n = 4$ in~\cite[Proposition~3.6]{SS17} and for even $n > 4$
in~\cite[Remark~3.7]{SS17}. Earlier results under additional hypotheses on~$u$ are surveyed
in~\cite[Theorem~3.3]{SS17}; see also~\cite[Theorem~1.1]{Kol10}.

Martino and Schikorra~\cite[Theorem~1.2]{MaSc26} and Schlagenhauf~\cite[Theorem~1.2]{Sch26} constructed bounded
discontinuous maps $U \to \R^{n+2}$ for $n \ge 3$ solving $n$-Laplace systems with antisymmetric
potentials in $\Lp{n}$. These examples concern a larger class of equations and give no
counterexample for $H$-systems; see~\cite{Sch26}.

Without additional hypotheses, the continuity of weak solutions for $n \ge 3$ is open even for
bounded Lipschitz~$H$; see~\cite[Problem~2.4]{SS17}.

Fusco, Kristensen, Leone, and Verde~\cite[Proposition~1.1]{FKLV18} proved that locally bounded weak
solutions are locally H{\"o}lder continuous when $H$ is continuous. When $H$ is locally Lipschitz,
they proved that such solutions are of class~$1$ and that their derivatives are locally H{\"o}lder
continuous. Thus, for continuous $H$, it suffices to show that $u$ is locally bounded.

The push-forward of the energy measure~$|\uD u|^n \LM^n$ by~$u$ is a measure~$\mu$
on~$\R^{n+1}$. Suppose that $\mu \, \oball ar \ge c$ for every point~$a \in \spt \mu$ far from
$u [\Bdry U]$, with fixed $r > 0$ and $c > 0$. Since
$\mu(\R^{n+1}) = \textint{}{}|\uD u|^n \ud \LM^n < \infty$, the support can then contain only
finitely many points at mutual distance at least $2r$ and far from $u [\Bdry U]$, and $u$ is
bounded; see~\ref{thm:bounded}. If $u$ is conformal, Mou and Yang~\cite[section~4.3]{MY96} obtain
such a~mass bound by testing the equation with $u - a$ multiplied by a cut-off function of
$|u - a|$, and integrating the resulting differential inequality for
$[ \rho \mapsto \mu \, \oball a\rho ]$. The differential inequality uses conformality only through
the inequality $|\uD u(x)^*(v)|^2 \le n^{-1} |\uD u(x)|^2 |v|^2$ for $v \in \R^{n+1}$. For general
$u$ one pushes forward $|\uD u|^{n-2} \uD u \circ \uD u^*$ by~$u$. This gives a linear functional
$\Sigma$ on the space of continuous compactly supported maps of $\R^{n+1}$ into
$\Hom(\R^{n+1},\R^{n+1})$, and the
$H$-system controls the values of $\Sigma$ on derivatives. The density $X$ of $\Sigma$ with respect
to $\mu$, given by $\Sigma(F) = \textint{}{}F \bullet X \ud \mu$, is a weighted average over the
fibres of~$u$ of the maps $|A|^{-2} A \circ A^*$ with $A = \uD u(x)$. If $A$ is conformal, this map
equals $n^{-1} \eqproject{\im A}$; see~\ref{mr:distortion}. Otherwise its largest eigenvalue
exceeds~$n^{-1}$, and the inequality used in~\cite{MY96} fails.

De Philippis, Gennaioli, Pigati, and Rindler~\cite[Theorem~1.1]{DGPR26} proved a very general
inequality of Michael--Simon type~\cite{MS73} for vector-valued measures whose divergence is again
a~measure. Their proof adapts the heat-flow method introduced by Gennaioli and Rindler~\cite{GR26},
who used it for the resolution of the Vanishing Mass Conjecture of Bouchitt{\'e} for first-order
operators. Our functional $\Sigma$ represents a vector-valued measure (see~\ref{mr:DGPR:language})
which meets all the requirements of~\cite[Theorem~1.1]{DGPR26} whenever the distortion
$|\uD u|^n / |\J u|$ of~$u$ is bounded. In particular $X$ takes values in a~compact convex set on
which the $n$-th singular value is bounded from below; see~\ref{lem:K}, \ref{lem:structure},
and~\ref{mr:DGPR}. This gives lower mass ratios in~\ref{lem:lower-mass}, boundedness
in~\ref{thm:bounded}, and continuity in~\ref{cor:continuity}.
\begin{theorem}[\protect{cf.~\ref{cor:continuity}}]
    If $H$ is bounded and continuous and the distortion of~$u$ is bounded on the set where $|u|$ is
    large, then $u$ is locally H{\"o}lder continuous.  If moreover $H$ is locally Lipschitz, then
    $u$ is of class~$1$ and its derivative is locally H{\"o}lder continuous.
\end{theorem}
For conformal $u$ the distortion equals $n^{n/2}$, and continuity is due to Mou and
Yang~\cite[Theorem~16]{MY96}. We do not know an earlier reference for bounded distortion greater
than~$n^{n/2}$. 

\section{Notation and the \texorpdfstring{$H$}{H}-system}
\label{sec:notation}

\begin{miniremark}
    \label{mr:notation}
    In this section we fix notation, define weak solutions and the distortion, and prove the
    estimate~\ref{lem:K} on which the argument rests. Throughout, $n \ge 2$ is an integer and
    $U = \oball 01 \subseteq \R^n$. The notation follows~\cite[pp.~669--671]{Fed69}. In particular
    $\natp$ denotes the set of positive integers, $\oball ar$ and $\cball ar$ are the open and the
    closed ball, $\LM^n$ and $\HM^k$ are the Lebesgue and the Hausdorff measure, $\mu \restrict A$
    is the restriction of $\mu$ to~$A$, and $\|T\|$ is the variation measure of a linear
    functional~$T$ on a space of continuous maps with compact support; see~\ref{mr:push}. The
    standard bases of $\R^n$ and $\R^{n+1}$ are $e_1, \ldots, e_n$ and $e_1, \ldots, e_{n+1}$.

    For linear maps $A$ and $B$ between Euclidean spaces with the same domain
    and the same codomain we set
    \begin{displaymath}
        A \bullet B = \trace (A^* \circ B) \,, \quad
        |A| = (A \bullet A)^{1/2} \,, \quad
        \|A\| = \sup \{ |A(v)| : |v| \le 1 \} \,.
    \end{displaymath}
    Then $\|A\| \le |A|$. If $A$ and $B$ are linear maps between Euclidean
    spaces and the domain of $A$ is the codomain of~$B$, then $|A \circ B|
    \le \|A\| |B|$.

    A vector space and its dual are not identified. For $\omega \in
    \Hom(\R^k,\R)$ and $b \in \R^N$ let $\omega \cdot b \in \Hom(\R^k,\R^N)$
    be given by $(\omega \cdot b)(w) = \omega(w) \, b$ for $w \in \R^k$; this is
    Allard's convention, cf.~\cite[2.2, 2.3]{All72}. If $o_1, \ldots, o_k$ is an
    orthonormal basis of $\R^k$ and $Z \in \Hom(\R^k,\R^N)$, then
    \begin{equation}
        \label{eq:notation:allard}
        (\omega \cdot b) \bullet Z
        = \textsum{i=1}{k} \omega(o_i) \, b \bullet Z(o_i)
        = \textsum{i=1}{k} \omega(o_i) \, Z^*(b) \bullet o_i
        = \omega \bigl( Z^*(b) \bigr) \,.
    \end{equation}
    If $\varphi$ is a real valued function of class~$1$ on an open subset
    of~$\R^k$, then $\grad \varphi(y) \in \R^k$ is defined by
    $\uD\varphi(y)w = \grad \varphi(y) \bullet w$ for $w \in \R^k$. If $g$ takes
    values in $\R^N$ and has class~$1$, then
    $\uD(\varphi g) = \varphi \uD g + \uD\varphi \cdot g$, and
    by~\eqref{eq:notation:allard}
    \begin{equation}
        \label{eq:notation:rank-one}
        \bigl( \uD\varphi(y) \cdot b \bigr) \bullet Z = b \bullet Z \bigl(
        \grad \varphi(y) \bigr)
        \quad \text{for $b \in \R^N$ and $Z \in \Hom(\R^k,\R^N)$} \,.
    \end{equation}
    We write $\cnt{1}_c(\R^k,\R^N)$ for the vector space of maps
    $g : \R^k \to \R^N$ of class~$1$ with compact support, normed by
    $\sup \im (|g| + |\uD g|)$.  If $X$ is a Euclidean space and $T \subseteq X$
    is a linear subspace, then we write $\project{T}$ for the orthogonal
    projection onto~$T$, cf.~\cite[T.1(9)]{Alm00}. In particular, if $\nu \in X$
    and $|\nu| = 1$, then $\eqproject{\R \nu} \in \Hom(X,X)$ denotes the
    orthogonal projection onto $\R \nu$, that is,
    $\eqproject{\R \nu}(x) = (x \bullet \nu) \nu$ for $x \in X$.

    For $A \in \Hom(\R^n,\R^{n+1})$ we have
    $\|\vwedge n A\| = | \vwedge n A | = |A e_1 \wedge \cdots \wedge A e_n|$,
    the norm on~$\vwedge n \R^{n+1}$ being that
    of~\cite[1.7.5]{Fed69}. By~\cite[1.7.5]{Fed69},
    $| \vwedge n A |^2 = \det (A^* \circ A)$; thus, $| \vwedge n A |$ is the $n$
    dimensional Jacobian of $A$ of~\cite[3.2.1]{Fed69}. Moreover
    $| \vwedge n A | = |A v_1 \wedge \cdots \wedge A v_n|$ for every orthonormal
    basis $v_1, \ldots, v_n$ of~$\R^n$. Let
    $\star : \vwedge p \R^{n+1} \to \vwedge {n+1-p} \R^{n+1}$ be the Hodge star
    operator denoted $*$ in~\cite[1.7.8]{Fed69}, defined with respect to the
    unit $(n+1)$ vector $e_1 \wedge \cdots \wedge e_{n+1}$; the symbol $*$ is
    reserved for adjoints.  Then $\star$ maps each element
    $e_{i_1} \wedge \cdots \wedge e_{i_p}$ of the standard orthonormal basis to
    the complementary one, up to the sign, and thus preserves the inner
    product. Moreover
    $\xi \wedge \star \eta = (\xi \bullet \eta) \, e_1 \wedge \cdots \wedge
    e_{n+1}$ for $\xi, \eta \in \vwedge p \R^{n+1}$, and
    $\star \star \xi = (-1)^{p(n+1-p)} \xi$. For $\zeta \in \vwedge n \R^{n+1}$
    and $w \in \R^{n+1}$, the choice $\eta = (-1)^n \star w$ gives
    $\star \eta = w$. Since $\star$ preserves the inner product,
    \begin{equation}
        \label{eq:notation:hodge}
        \zeta \wedge w = (-1)^n (\zeta \bullet \star w) \, e_1 \wedge \cdots
        \wedge e_{n+1} = (w \bullet \star \zeta) \, e_1 \wedge \cdots \wedge
        e_{n+1} \,.
    \end{equation}
    Define
    \begin{displaymath}
        \J (A) = \star (A e_1 \wedge \cdots \wedge A e_n)
        \quad \text{for $A \in \Hom(\R^n,\R^{n+1})$} \,.
    \end{displaymath}
    Then $|\J (A)| = | \vwedge n A |$ and $\J (A) \bullet A(v) = 0$ for
    $v \in \R^n$. Indeed, $\star$ preserves the inner product, and the $n+1$
    vectors $A e_1, \ldots, A e_n, A(v)$ lie in $\im A$, which has dimension at
    most~$n$, so that $A e_1 \wedge \cdots \wedge A e_n \wedge A(v) = 0$; now
    \eqref{eq:notation:hodge} applies.  If $A \ne 0$, we set
    \begin{displaymath}
        Y(A) = |A|^{-2} A \circ A^* \in \Hom(\R^{n+1},\R^{n+1}) \,.
    \end{displaymath}

    For $X \in \Hom(\R^d,\R^d)$ and $k \in \{1, \ldots, d\}$ let
    \begin{displaymath}
        \sigma_k(X) = \sup \bigl\{ \inf \{ |X(w)| : w \in W,\ |w| = 1 \}
        : W \in \grass dk \bigr\} \,.
    \end{displaymath}
    By the Courant--Fischer principle~\cite[Problem~III.6.1]{Bha97}, $\sigma_k(X)$ is
    the $k$-th largest singular value of~$X$, see~\cite[section~I.2]{Bha97}, and it is
    the $k$-th largest eigenvalue of $X$ if $X$ is symmetric and non-negative. In particular $\sigma_1(X) = \|X\|$.
\end{miniremark}

\begin{miniremark}
    \label{mr:separable}
    Suppose $k, N \in \natp$. There are
    $g_1, g_2, \ldots \in \cnt{1}_c(\R^k,\R^N)$ such that each
    $g \in \cnt{1}_c(\R^k,\R^N)$ is the limit of a subsequence of
    $g_1, g_2, \ldots$. Indeed, for each~$j \in \natp$, the compact metric space
    $\cball 0j \times \{1, \ldots, (k+1)N\}$ has a countable base, and the space
    of continuous real valued functions on it, normed by the supremum norm, is
    separable by~\cite[V.8.4]{DS58}. Mapping $g$ onto the function on
    $\cball 0j \times \{1, \ldots, (k+1)N\}$ whose values at
    $(y,1), \ldots, (y,(k+1)N)$ are the coordinates of
    $g(y), \uD g(y)(e_1), \ldots, \uD g(y)(e_k)$ with respect to the standard
    basis of~$\R^N$, where $e_1, \ldots, e_k$ is the standard basis of~$\R^k$,
    we embed the subspace
    $\cnt{1}_c(\R^k,\R^N) \cap \{ g : \spt g \subseteq \cball 0j \}$ linearly
    and homeomorphically into that space of functions. Hence,
    by~\cite[I.6.12]{DS58}, this subspace has a countable dense subset. The
    union of these subsets over~$j$ is countable, and each
    $g \in \cnt{1}_c(\R^k,\R^N)$ belongs to the closure of one of them.
\end{miniremark}

\begin{definition}
    \label{def:jacobian}
    Given a function $v$ mapping an open subset of~$\R^n$ into~$\R^{n+1}$
    we set
    \begin{displaymath}
        \J v(x) = \J (\uD v(x))
        \quad \text{for $x \in \dmn \uD v$} \,.
    \end{displaymath}
\end{definition}

\begin{definition}
    \label{def:hsys}
    Suppose $0 \le h < \infty$ and $H : \R^{n+1} \to \R$ is a Borel
    function with $|H(y)| \le h$ for $y \in \R^{n+1}$. We say that $u$ is a
    \emph{weak solution of the $H$-system} if $u \in W^{1,n}(U,\R^{n+1})$ and
    \begin{equation}
        \label{eq:hsys}
        \textint{U}{} |\uD u|^{n-2} \uD u \bullet \uD\psi \ud \LM^n
        = \textint{U}{} H(u(x)) \, \J u(x) \bullet \psi(x) \ud \LM^n x
    \end{equation}
    whenever $\psi \in W^{1,n}_0(U,\R^{n+1}) \cap \Lp{\infty}(U,\R^{n+1})$.
\end{definition}

\begin{lemma}
    \label{lem:hadamard}
    If $A \in \Hom(\R^n,\R^{n+1})$, then
    \begin{gather}
        \label{eq:hadamard:amgm}
        |\J (A)| \le n^{-n/2} |A|^n \,,
        \\
        \label{eq:hadamard:smallest}
        |A(v)| \ge |\J (A)| \, |A|^{1-n} |v|
        \quad \text{for $v \in \R^n$, provided $A \ne 0$} \,.
    \end{gather}
\end{lemma}

\begin{proof}
    Let $v_1, \ldots, v_n$ be an orthonormal basis of $\R^n$. Hadamard's
    inequality gives $|\J (A)| = |A v_1 \wedge \cdots \wedge A
    v_n| \le |A v_1| \cdots |A v_n|$. Since $|A v_1|^2 + \cdots + |A v_n|^2
    = |A|^2$, the inequality between the geometric and the arithmetic mean
    yields~\eqref{eq:hadamard:amgm}. Suppose $A \ne 0$ and $|v| = 1$. We
    choose $v_1 = v$ and use $|A v_i| \le \|A\| \le |A|$ for $i \ge 2$ to
    obtain $|\J (A)| \le |A(v)| \, |A|^{n-1}$. This
    is~\eqref{eq:hadamard:smallest}.
\end{proof}

\begin{miniremark}
    \label{mr:distortion}
    Let $A \in \Hom(\R^n,\R^{n+1})$ with $\vwedge n A \ne 0$. We call
    $|A|^n / |\J (A)|$ the \emph{distortion} of~$A$. By the proof
    of~\ref{lem:hadamard}, it is at least $n^{n/2}$, with equality if and
    only if $A^* \circ A$ is a multiple of the identity. In this case
    $Y(A) = n^{-1} \eqproject{\im A}$.
\end{miniremark}

\begin{definition}
    \label{def:K}
    For $n^{n/2} \le \kappa < \infty$ let $\Psi_\kappa$ be the closure
    in $\Hom(\R^{n+1},\R^{n+1})$ of the convex hull of
    \begin{displaymath}
        \bigl\{ Y(A) : A \in \Hom(\R^n,\R^{n+1}),\
        0 < |A|^n \le \kappa |\J (A)| \bigr\} \,.
    \end{displaymath}
    Then $Y(A) \in \Psi_\kappa$ whenever $A \in \Hom(\R^n,\R^{n+1})$ and
    $0 < |A|^n \le \kappa |\J (A)|$.
\end{definition}

\begin{lemma}
    \label{lem:K}
    If $n^{n/2} \le \kappa < \infty$, then $\Psi_\kappa$ is a non-empty
    compact convex set, and each $X \in \Psi_\kappa$ is symmetric and
    satisfies
    \begin{equation}
        \label{eq:K}
        Xw \bullet w \ge 0 \text{ for $w \in \R^{n+1}$} \,, \quad
        \trace X = 1 \,, \quad
        \sigma_1(X) \le 1 \,, \quad
        \sigma_n(X) \ge (2 \kappa^2)^{-1} \,.
    \end{equation}
\end{lemma}

\begin{proof}
    Call $N \in \Hom(\R^{n+1},\R^{n+1})$ \emph{admissible} if $N$ is
    symmetric, $Nw \bullet w \ge 0$ for $w \in \R^{n+1}$, and $\trace N = 1$.
    Convex combinations and limits of admissible maps are admissible, and $|N|
    \le \trace N = 1$ for admissible $N$. For a symmetric $X \in
    \Hom(\R^{n+1},\R^{n+1})$ and an admissible $N$ consider the condition
    \begin{equation}
        \label{eq:K:S}
        Xw \bullet w \ge \kappa^{-2} \bigl( |w|^2 - Nw \bullet w \bigr)
        \quad \text{for $w \in \R^{n+1}$} \,.
    \end{equation}

    Suppose $A \in \Hom(\R^n,\R^{n+1})$ and
    $0 < |A|^n \le \kappa |\J (A)|$. Clearly $\dim \im A = n$. Let
    \begin{displaymath}
        \nu \in (\im A)^{\perp} \,,
        \quad
        |\nu| = 1 \,,
        \quad
        P = \eqproject{\im A} \,,
        \quad \text{and} \quad
        s = |\J (A)| \, |A|^{1-n} \,.
    \end{displaymath}
    If $z \in \im A$, then $z = A(v)$ for some $v \in \R^n$, $|z| \ge s|v|$
    by~\ref{lem:hadamard}\eqref{eq:hadamard:smallest}, and
    $|A^*(z)| \, |v| \ge A^*(z) \bullet v = |z|^2$; hence $|A^*(z)| \ge s
    |z|$. Since $A^*(\nu) = 0$, $s |A|^{-1} \ge \kappa^{-1}$, and
    $|P(w)|^2 = |w|^2 - (w \bullet \nu)^2$, we infer
    \begin{displaymath}
        Y(A)w \bullet w
        = |A|^{-2} |A^*(P(w))|^2
        \ge \kappa^{-2} \bigl( |w|^2 - (w \bullet \nu)^2 \bigr)
        \quad \text{for $w \in \R^{n+1}$} \,.
    \end{displaymath}
    Since $\eqproject{\R \nu}$ is admissible and $\eqproject{\R \nu}(w)
    \bullet w = (w \bullet \nu)^2$, the condition~\eqref{eq:K:S} holds for $X
    = Y(A)$ and $N = \eqproject{\R \nu}$. Moreover $Y(A)$ is symmetric,
    $Y(A)w \bullet w \ge 0$ for $w \in \R^{n+1}$, and $\trace Y(A) = |A|^{-2}
    \trace (A \circ A^*) = 1$.

    Let $X$ be a convex combination of finitely many such maps $Y(A_i)$ with
    weights $t_i$, and let $N$ be the convex combination of the
    corresponding projections $N_i = \eqproject{\R \nu_i}$ with the same
    weights. Then $N$ is admissible, $X$ is symmetric, $Xw \bullet w \ge
    0$ for $w \in \R^{n+1}$, $\trace X = 1$, and~\eqref{eq:K:S} holds. If $X
    \in \Psi_\kappa$, then $X$ is the limit of a sequence of such convex
    combinations, with corresponding admissible maps. These maps are
    bounded, and a subsequence converges to some admissible $N$. Hence $X$
    is symmetric, $Xw \bullet w \ge 0$ for $w \in \R^{n+1}$, $\trace X =
    1$, and~\eqref{eq:K:S} holds. The set $\Psi_\kappa$ is bounded, since
    $|X| \le \trace X$ for symmetric positive-semidefinite~$X$, closed by definition,
    and convex, as the closure of a convex set. The inclusion $A$ of $\R^n$
    into $\R^{n+1}$ satisfies $|A|^n = n^{n/2} = n^{n/2} |\J (A)|$;
    hence $Y(A) \in \Psi_\kappa \ne \varnothing$.

    Let $X \in \Psi_\kappa$ and let $N$ be an admissible map as
    in~\eqref{eq:K:S}. The largest eigenvalue of $X$ is at most $\trace X = 1$,
    so $\sigma_1(X) \le 1$. Let $e$ be a unit eigenvector of $N$ for its largest
    eigenvalue $\lambda_1$, and let $W = \{ w : w \bullet e = 0 \}$.  Then
    $N[W] \subseteq W$, and every eigenvalue of $N|W$ is at most
    $\min \{ \lambda_1, 1 - \lambda_1 \} \le 1/2$, because the eigenvalues of
    $N$ are non-negative with sum~$1$. Hence $Nw \bullet w \le 1/2$ and
    $|Xw| \ge Xw \bullet w \ge (2 \kappa^2)^{-1}$ for $w \in W$ with $|w| = 1$.
    Since $W \in \grass{n+1}{n}$, we obtain $\sigma_n(X) \ge (2 \kappa^2)^{-1}$.
\end{proof}

\section{Push-forward functionals of a solution}
\label{sec:stress}

\begin{miniremark}
    \label{mr:stress}
    In this section we push forward $|\uD u|^{n-2} \uD u \circ \uD u^*$ by~$u$ to define a linear
    functional on continuous maps with compact support in the target space~$\R^{n+1}$
    (a.k.a. vector-valued measure; see~\ref{mr:DGPR:language}). We~compute this functional on the
    derivatives of maps whose support does not meet the image of the boundary sphere, and we
    represent the functional by integration when the distortion is bounded. Throughout the section,
    $0 \le h < \infty$, and $H$ and $u$ are as in~\ref{def:hsys}. We fix Borel functions
    $u : U \to \R^{n+1}$ and $\uD u : U \to \Hom(\R^n,\R^{n+1})$ representing $u$ and its weak
    derivative.
\end{miniremark}

\begin{miniremark}
    \label{mr:push}
    Suppose $d \in \natp$ and $W$ is a finite dimensional Euclidean space. As
    in~\cite[2.5.14]{Fed69}, $\ccspace{\R^d}$ is the space of continuous
    functions $f : \R^d \to \R$ with compact support, and
    $\ccspace{\R^d}^+ = \ccspace{\R^d} \cap \{ f : f \ge 0 \}$. Similarly,
    $\ccspace{\R^d,W}$ is the space of continuous maps of $\R^d$ into $W$ with
    compact support. Suppose $T : \ccspace{\R^d,W} \to \R$ is a linear
    functional such that
    \begin{displaymath}
        \|T\|(f) = \sup \bigl\{ T(\xi) : \xi \in \ccspace{\R^d,W},\
        |\xi(y)| \le f(y) \text{ for $y \in \R^d$} \bigr\} < \infty
        \quad \text{for $f \in \ccspace{\R^d}^+$} \,.
    \end{displaymath}
    We apply~\cite[2.5.12]{Fed69} in the form for $L = \ccspace{X}$ given in
    the last paragraph of~\cite[2.5.14]{Fed69}, with $\R^d$,
    $\ccspace{\R^d,W}$, $W$ and the norm of $W$ in place of $X$, $\Omega$,
    $E$ and~$\nu$, as in~\cite[4.1.5]{Fed69}. Hence $\|T\|$ is a monotone
    Daniell integral on $\ccspace{\R^d}^+$. As in~\cite[4.1.5]{Fed69}, we
    denote by $\|T\|$ also the Radon measure over $\R^d$ associated with it,
    and call it the \emph{variation measure} of~$T$. Moreover, there is an
    $\|T\|$ measurable map $\eta$ with values in $\Hom(W,\R)$ such that
    $\|\eta(y)\| = 1$ for $\|T\|$ almost all $y$ and
    \begin{equation}
        \label{eq:stress:polar}
        T(\xi) = \textint{}{}\eta(y)(\xi(y)) \ud \|T\| y
        \quad \text{for $\xi \in \ccspace{\R^d,W}$} \,.
    \end{equation}
    Since $\|T\|(f) = \textint{}{}f \ud \|T\|$ for $f \in \ccspace{\R^d}^+$, we
    have $T(\xi) \le \textint{}{}|\xi| \ud \|T\|$ for
    $\xi \in \ccspace{\R^d,W}$.  The functions $f_j \in \ccspace{\R^d}^+$ given
    by $f_j(y) = \min \{ 1, \max \{ 0, j - |y| \} \}$ for $y \in \R^d$ and
    $j \in \natp$ satisfy $f_j \le f_{j+1} \le 1$ and $f_j(y) = 1$ whenever
    $|y| \le j-1$.  By~\cite[2.4.7]{Fed69},
    $\|T\|(\R^d) = \lim_{j \to \infty} \|T\|(f_j)$.  Since $f_j \le 1$, we infer
    \begin{equation}
        \label{eq:stress:variation}
        \begin{gathered}
            |T(\xi)| \le \textint{}{}|\xi| \ud \|T\|
            \quad \text{for $\xi \in \ccspace{\R^d,W}$} \,,
            \\
            \|T\|(\R^d) = \sup \bigl\{ T(\xi) : \xi \in \ccspace{\R^d,W},\
            |\xi(y)| \le 1 \text{ for $y \in \R^d$} \bigr\} \,.
        \end{gathered}
    \end{equation}
\end{miniremark}

\begin{definition}
    \label{def:push}
    For $0 < \rho < 1$, $f \in \ccspace{\R^{n+1}}$,
    $F \in \ccspace{\R^{n+1},\Hom(\R^{n+1},\R^{n+1})}$, and
    $g \in \ccspace{\R^{n+1},\R^{n+1}}$ we define
    \begin{gather*}
        \mu_{\rho}(f) = \textint{\oball 0\rho}{} f(u(x)) \, |\uD u(x)|^n \ud \LM^n x
        \,,
        \\
        \Sigma_\rho(F) = \textint{\oball 0\rho}{} F(u(x)) \bullet
        \bigl( |\uD u|^{n-2} \uD u \circ \uD u^* \bigr)(x) \ud \LM^n x \,,
        \\
        \tau_\rho(g) = \textint{\oball 0\rho}{}
        H(u(x)) \, \J u(x) \bullet g(u(x)) \ud \LM^n x \,.
    \end{gather*}
\end{definition}

\begin{miniremark}
    \label{mr:push:measure}
    The linear functionals in~\ref{def:push} do not depend on the choice of the
    representatives. Since $u$ and $\uD u$ are Borel maps, the integrands
    in~\ref{def:push} are Borel functions. Since $\mu_{\rho}(f) \ge 0$ for
    $f \in \ccspace{\R^{n+1}}^+$, the functional $\mu_{\rho}$ is monotone,
    and~\cite[2.5.13]{Fed69} shows that it is a Daniell integral. Thus
    $\mu_{\rho}$ is a monotone Daniell integral in the sense
    of~\cite[2.5.6]{Fed69}. We denote with the same symbol $\mu_{\rho}$ also the
    $\ccspace{\R^{n+1}}$ regular measure associated with it. Since the topology
    of $\R^{n+1}$ has a countable base, it is a Radon measure
    by~\cite[2.5.14]{Fed69}, and $\mu_{\rho}(f) = \textint{}{}f \ud \mu_{\rho}$
    for $f \in \ccspace{\R^{n+1}}$.

    Suppose $V \subseteq \R^{n+1}$ is open, let $f_j$ be as in~\ref{mr:push}
    with $d = n+1$, and let $\dist(y,\varnothing) = \infty$ for $y \in \R^{n+1}$. The
    functions $\varphi_j : \R^{n+1} \to \R$ given by
    \begin{displaymath}
        \varphi_j(y) = \min \bigl\{ f_j(y), j \dist(y, \R^{n+1} \without V)
        \bigr\}
        \quad \text{for $y \in \R^{n+1}$ and $j \in \natp$}
    \end{displaymath}
    belong to $\ccspace{\R^{n+1}}^+$ and increase to the characteristic
    function of $V$ as $j \to \infty$. We replace $f$ by $\varphi_j$ in the
    identity
    \begin{displaymath}
        \textint{}{}f \ud \mu_{\rho} = \textint{\oball 0\rho}{} f(u(x)) \, |\uD u(x)|^n \ud
        \LM^n x
    \end{displaymath}
    and apply~\cite[2.4.7]{Fed69} to both sides. This gives
    \begin{equation}
        \label{eq:stress:borel}
        \mu_{\rho}(V) = \textint{\oball 0\rho \cap u^{-1}[V]}{} |\uD u|^n \ud \LM^n \,. 
    \end{equation}
    Both sides of~\eqref{eq:stress:borel} are finite and
    additive on Borel sets; hence~\eqref{eq:stress:borel} also holds for
    closed sets~$E$ in place of~$V$. Suppose $E \subseteq \R^{n+1}$ is a Borel set.
    By~\cite[2.2.2(1)]{Fed69}, $\mu_{\rho}(E)$ is the supremum of the numbers
    $\mu_{\rho}(C)$ corresponding to closed sets $C \subseteq E$. Hence
    $\mu_{\rho}(E)$ does not exceed the right-hand side
    of~\eqref{eq:stress:borel}. The same holds for $\R^{n+1} \without E$ in
    place of~$E$, and~\eqref{eq:stress:borel} with $E = \R^{n+1}$
    yields~\eqref{eq:stress:borel} for every Borel set $E \subseteq
    \R^{n+1}$.

    Since $|A \circ A^*| \le \|A\| |A| \le |A|^2$ for linear maps~$A$
    and~\ref{lem:hadamard}\eqref{eq:hadamard:amgm} holds, we have
    \begin{equation}
        \label{eq:stress:bounds}
        |\Sigma_\rho(F)| \le \textint{}{}|F| \ud \mu_{\rho} \,, \qquad
        |\tau_\rho(g)| \le h n^{-n/2} \textint{}{}|g| \ud \mu_{\rho}
    \end{equation}
    for $F \in \ccspace{\R^{n+1},\Hom(\R^{n+1},\R^{n+1})}$ and $g \in
    \ccspace{\R^{n+1},\R^{n+1}}$.
\end{miniremark}

\begin{miniremark}
    \label{mr:sphere:strategy}
    The proof of~\ref{lem:sphere} follows that of~\cite[Theorem~4.2]{Men16}. We test the
    system with $\omega(|x|) \, g(u(x))$, where $\omega$ is a primitive of a
    test function $\varrho$ of the radius. Fubini's theorem then
    gives~\ref{lem:sphere}\eqref{eq:sphere:test}, and integration in polar
    coordinates turns it into an identity for almost every radius. The
    continuity of the traces of $u$ on almost all spheres follows
    from~\cite[Chap.~2, Sec.~4.3, Lemma~1]{GMS98}.
\end{miniremark}

\begin{lemma}
    \label{lem:sphere}
    Suppose $J = \{ s : 0 < s < 1 \}$. Then
    \begin{equation}
        \label{eq:sphere:test}
        \textint{J}{} \varrho(s) \bigl( \Sigma_s(\uD g) - \tau_s(g) \bigr) \ud \LM^1 s
        = \textint{U}{} \varrho(|x|) \, |\uD u(x)|^{n-2} \, g(u(x)) \bullet \uD
        u(x)(|x|^{-1} x) \ud \LM^n x
    \end{equation}
    whenever $\varrho \in \dspace J\R$ and $g \in
    \cnt{1}_c(\R^{n+1},\R^{n+1})$. Moreover, for $\LM^1$ almost all $\rho \in
    J$ there is a continuous map $v : \Bdry{\oball 0\rho} \to
    \R^{n+1}$ such that
    \begin{gather}
        u(x) = v(x) \quad \text{for $\HM^{n-1}$ almost all $x \in \Bdry{\oball 0\rho}$} \,,
        \quad
        \textint{\Bdry{\oball 0\rho}}{} |\uD u|^n \ud \HM^{n-1} < \infty \,,
        \\ \text{and} \quad 
        \label{eq:sphere}
        \Sigma_\rho(\uD g) = \tau_\rho(g) + \beta_\rho(g)
    \end{gather}
    whenever $g \in \cnt{1}_c(\R^{n+1},\R^{n+1})$. Here $\beta_\rho$ is the
    linear functional on $\ccspace{\R^{n+1},\R^{n+1}}$ given by
    \begin{displaymath}
        \beta_\rho(g)
        = \textint{\Bdry{\oball 0\rho}}{}
        |\uD u(x)|^{n-2} \, g(v(x)) \bullet \uD u(x)(|x|^{-1} x) \ud
        \HM^{n-1} x \,.
    \end{displaymath}
\end{lemma}

\begin{proof}
    Suppose $\varrho \in \dspace J\R$ and $g \in
    \cnt{1}_c(\R^{n+1},\R^{n+1})$. Define $\omega : \R \to \R$ by
    \begin{displaymath}
        \omega(t) = \textint{\sup \{ t, 0 \}}{1} \varrho \ud \LM^1
        \quad \text{for $t \in \R$} \,.
    \end{displaymath}
    Then $\omega$ is of class~$\infty$, $\omega'(t) = - \varrho(t)$ for
    $t \in J$, $\omega(t) = 0$ for $t \ge \sup \spt \varrho$, and $\omega$ is
    constant on $\{ t : t \le \inf \spt \varrho \}$. Hence the function mapping
    $x \in U$ onto $\omega(|x|)$ is of class~$\infty$ with compact support
    in~$U$, and its gradient at $x$ equals $- \varrho(|x|) |x|^{-1} x$ if
    $x \ne 0$ and $0$ if $x = 0$. One readily verifies that
    $g \circ u \in W^{1,n}(U,\R^{n+1})$ and
    $\uD(g \circ u) = \uD g(u) \circ \uD u$; cf.~\cite[Lemma~7.5]{GT01}.  Therefore
    the map $\psi$ given by $\psi(x) = \omega(|x|) \, g(u(x))$ for $x \in U$
    belongs to $W^{1,n}_0(U,\R^{n+1}) \cap \Lp{\infty}(U,\R^{n+1})$.  Since
    $\uD u \circ \uD u^*$ is symmetric,
    \begin{displaymath}
        \uD u \bullet \bigl( \uD g(u) \circ \uD u \bigr)
        = \trace \bigl( \uD g(u) \circ \uD u \circ \uD u^* \bigr)
        = \uD g(u) \bullet \bigl( \uD u \circ \uD u^* \bigr) \,.
    \end{displaymath}
    Testing~\ref{def:hsys}\eqref{eq:hsys} with $\psi$ and
    using~\ref{mr:notation}\eqref{eq:notation:rank-one}, we obtain
    \begin{multline*}
        \textint{U}{} \omega(|x|) \, \uD g(u(x)) \bullet \bigl( |\uD u|^{n-2} \uD u
        \circ \uD u^* \bigr)(x) \ud \LM^n x
        \\
        - \textint{U}{} \varrho(|x|) \, |\uD u(x)|^{n-2} \, g(u(x)) \bullet \uD
        u(x)(|x|^{-1} x) \ud \LM^n x
        \\
        = \textint{U}{} \omega(|x|) \, H(u(x)) \, \J u(x) \bullet g(u(x))
        \ud \LM^n x \,.
    \end{multline*}
    Since $\omega(|x|) = \textint{J \cap \{ s : s > |x| \}}{} \varrho \ud \LM^1$
    for $x \in U$, Fubini's theorem~\cite[2.6.2]{Fed69} transforms the first
    and the last integral into $\textint{J}{} \varrho(s) \, \Sigma_s(\uD g) \ud \LM^1
    s$ and $\textint{J}{} \varrho(s) \, \tau_s(g) \ud \LM^1 s$. This
    proves~\eqref{eq:sphere:test}.

    If $g \in \cnt{1}_c(\R^{n+1},\R^{n+1})$, we define
    \begin{displaymath}
        f_g(t) = \textint{\Bdry{\oball 0t}}{} |\uD u(x)|^{n-2} \, g(u(x)) \bullet
        \uD u(x)(|x|^{-1} x) \ud \HM^{n-1} x
    \end{displaymath}
    whenever $t \in J$ and $\textint{\Bdry{\oball 0t}}{} |\uD u|^{n-1} \ud
    \HM^{n-1} < \infty$. Since $U$ is bounded, $|\uD u|^{n-1}$ is $\LM^n
    \restrict U$ summable. By the integration in polar
    coordinates~\cite[3.2.13]{Fed69}, $f_g$ is $\LM^1 \restrict J$ summable,
    and the right-hand side of~\eqref{eq:sphere:test} equals $\textint{J}{} \varrho
    f_g \ud \LM^1$. The function mapping $t \in J$ onto $\Sigma_t(\uD g) -
    \tau_t(g)$ is bounded and continuous. Hence~\eqref{eq:sphere:test}
    shows that the regularisations, defined in~\cite[(7.13)]{GT01}, of the
    function $[ J \ni t \mapsto \Sigma_t(\uD g) - \tau_t(g) - f_g(t) ]$
    vanish on each compact subset of $J$ for small radii. Therefore
    \cite[Lemma~7.2]{GT01} yields
    \begin{equation}
        \label{eq:sphere:ae}
        \Sigma_t(\uD g) = \tau_t(g) + f_g(t)
        \quad \text{for $\LM^1$ almost all $t \in J$} \,.
    \end{equation}

    Let $\tilde u$ be the approximately continuous representative of~$u$:
    if $x \in U$ and $u$ has an approximate limit at~$x$, see~\cite[2.9.12]{Fed69},
    then $\tilde u(x)$ is this limit. By~\cite[Chap.~2, Sec.~4.3, Lemma~1]{GMS98} with $p = n$, applied
    to the maps sending $x \in U$ onto $(u(x) \bullet e_i) e_1 \in \R^n$ for
    $i \in \{1, \ldots, n+1\}$, the map $\tilde u$ is defined and continuous
    on $\Bdry{\oball 0t}$ for $\LM^1$ almost all $t \in J$. Moreover $u(x) =
    \tilde u(x)$ for $\LM^n$ almost all $x \in U$ by~\cite[2.9.13]{Fed69}.
    Hence the integration in polar coordinates~\cite[3.2.13]{Fed69} shows
    that $u = \tilde u$ at $\HM^{n-1}$ almost all points of $\Bdry{\oball
    0t}$ and $\textint{\Bdry{\oball 0t}}{} |\uD u|^n \ud \HM^{n-1} < \infty$ for
    $\LM^1$ almost all $t \in J$.

    By~\ref{mr:separable}, there are $g_1, g_2, \ldots \in
    \cnt{1}_c(\R^{n+1},\R^{n+1})$ such that each $g \in
    \cnt{1}_c(\R^{n+1},\R^{n+1})$ is the limit of a subsequence of $g_1, g_2,
    \ldots$. Let $T$ be the set of $t \in J$ such
    that $\tilde u$ is defined and continuous on $\Bdry{\oball 0t}$,
    $u = \tilde u$ at $\HM^{n-1}$ almost all points of $\Bdry{\oball 0t}$,
    $\textint{\Bdry{\oball 0t}}{} |\uD u|^n \ud \HM^{n-1} < \infty$, and
    $\Sigma_t(\uD g_k) = \tau_t(g_k) + f_{g_k}(t)$ for $k \in \natp$. Then
    $\LM^1(J \without T) = 0$ by~\eqref{eq:sphere:ae}. Suppose $\rho \in T$, let
    $S = \Bdry{\oball 0\rho}$, and let $v = \tilde u | S$. Since
    $\HM^{n-1}(S) < \infty$, H{\"o}lder's inequality~\cite[2.4.14]{Fed69} gives
    $\textint{S}{} |\uD u|^{n-1} \ud \HM^{n-1} < \infty$.

    Since $u = v$ at $\HM^{n-1}$ almost all points of~$S$, we have
    $f_{g_k}(\rho) = \beta_\rho(g_k)$, and $\rho \in T$ gives~\eqref{eq:sphere}
    for $g = g_k$, $k \in \natp$. Since $|\uD u(x)(|x|^{-1} x)| \le |\uD u(x)|$,
    we have
    $|\beta_\rho(g)| \le \sup |g| \textint{S}{} |\uD u|^{n-1} \ud \HM^{n-1}$ for
    $g \in \ccspace{\R^{n+1},\R^{n+1}}$. Combining this
    with~\ref{mr:push:measure}\eqref{eq:stress:bounds}, we obtain
    \begin{gather*}
        |\Sigma_\rho(\uD g)| \le \mu_{\rho}(\R^{n+1}) \sup |\uD g| \,,
        \\
        |\tau_\rho(g) + \beta_\rho(g)|
        \le \Bigl( h n^{-n/2} \mu_{\rho}(\R^{n+1}) + \textint{S}{} |\uD u|^{n-1} \ud
        \HM^{n-1} \Bigr) \sup |g| \,.
    \end{gather*}
    Thus both sides of~\eqref{eq:sphere} depend continuously on $g \in
    \cnt{1}_c(\R^{n+1},\R^{n+1})$, and~\eqref{eq:sphere} holds for all $g \in
    \cnt{1}_c(\R^{n+1},\R^{n+1})$ by the choice of $g_1, g_2, \ldots$.
\end{proof}

\begin{definition}
    \label{def:regular}
    We call $\rho$ \emph{regular} if $0 < \rho < 1$ and there is a map $v$
    as in the second part of~\ref{lem:sphere}. In this case we set $K_\rho =
    v[\Bdry{\oball 0\rho}]$.
\end{definition}

\begin{miniremark}
    \label{mr:regular}
    If $\rho$ is regular, then the map $v$ in~\ref{def:regular} is unique,
    because $\HM^{n-1} \restrict \Bdry{\oball 0\rho}$ is positive on non-empty
    relatively open subsets of $\Bdry{\oball 0\rho}$, and $K_\rho$ is a
    compact set. If $\rho$ is regular, $g \in \cnt{1}_c(\R^{n+1},\R^{n+1})$
    and $\spt g \cap K_\rho = \varnothing$, then $g(v(x)) = 0$ for $x \in
    \Bdry{\oball 0\rho}$, hence $\beta_\rho(g) = 0$, and~\ref{lem:sphere}
    yields
    \begin{equation}
        \label{eq:regular}
        \Sigma_\rho(\uD g) = \tau_\rho(g) \,.
    \end{equation}
\end{miniremark}

\begin{miniremark}
    \label{mr:structure:strategy}
    The proof of~\ref{lem:structure} applies the area formula for Sobolev
    maps of Haj{\l}asz~\cite[Theorem~12]{Haj00} to~$u$. The bound on the
    distortion implies that the resulting density of $\Sigma_\rho$ with
    respect to $\mu_{\rho}$ takes values in~$\Psi_\kappa$.
\end{miniremark}

\begin{lemma}
    \label{lem:structure}
    Suppose $0 \le R_0 < \infty$, $n^{n/2} \le \kappa < \infty$,
    \begin{equation}
        \label{eq:structure:hyp}
        |\uD u(x)|^n \le \kappa |\J u(x)|
        \quad \text{for $\LM^n$ almost all $x \in U$ with $|u(x)| > R_0$} \,,
    \end{equation}
    and $0 < \rho < 1$. Then there exist a Borel set $M \subseteq \R^{n+1}
    \without \cball 0{R_0}$ with $\HM^n(M) < \infty$, an $\HM^n \restrict M$
    measurable function $\theta : M \to [n^{n/2}, \infty)$, and an $\HM^n
    \restrict M$ measurable map $X : M \to \Psi_\kappa$ such that
    \begin{equation}
        \label{eq:structure}
        \mu_{\rho}(E) = \textint{E \cap M}{} \theta \ud \HM^n \,, \qquad
        \Sigma_\rho(F) = \textint{M}{} \theta \, F \bullet X \ud \HM^n
    \end{equation}
    whenever $E \subseteq \R^{n+1} \without \cball 0{R_0}$ is a Borel set,
    $F \in \ccspace{\R^{n+1},\Hom(\R^{n+1},\R^{n+1})}$, and $\spt F \cap
    \cball 0{R_0} = \varnothing$.
\end{lemma}

\begin{proof}
    Let $P$ be the set of $x \in \oball 0\rho$ with $|u(x)| > R_0$, $\uD u(x)
    \ne 0$, and $|\uD u(x)|^n \le \kappa |\J u(x)|$. Then $\vwedge n
    \uD u(x) \ne 0$ for $x \in P$, and~\eqref{eq:structure:hyp} gives
    \begin{equation}
        \label{eq:structure:P}
        \text{$\uD u(x) = 0$ or $x \in P$}
        \quad \text{for $\LM^n$ almost all $x \in \oball 0\rho$ with $|u(x)|
          > R_0$} \,.
    \end{equation}
    Choose $\zeta \in \dspace U\R$ with $\zeta(x) = 1$ for $|x| \le \rho$, and
    let $f : \R^n \to \R^{n+1}$ be the map that equals $\zeta u$ on $U$ and $0$
    on $\R^n \without U$. Then $f \in W^{1,n}(\R^n,\R^{n+1})$, $f(x) = u(x)$ for
    $x \in \oball 0\rho$, and $\uD f(x) = \uD u(x)$ for $\LM^n$ almost all
    $x \in \oball 0\rho$. Recalling~\ref{mr:notation} and
    employing~\cite[Theorem~12]{Haj00} with $p = n$, there is an $\LM^n$ null
    set $N \subseteq \R^n$ such that
    \begin{displaymath}
        \textint{}{} g \, |\J f| \ud \LM^n
        = \textint{}{}\textsum{x \in f^{-1}[\{y\}] \without N}{} g(x) \ud \HM^n y
    \end{displaymath}
    whenever $g : \R^n \to \R$ is either non-negative and $\LM^n$ measurable
    or $\LM^n$ measurable with $g \, |\J f|$ $\LM^n$ summable.
    Let $A = P \without N$. Applying this to the functions $g$ that equal
    $\psi$ on $A$ and $0$ elsewhere, and recalling that $f = u$ on $A$, we
    obtain
    \begin{equation}
        \label{eq:structure:area}
        \textint{A}{} \psi \, |\J u| \ud \LM^n
        = \textint{}{}\textsum{x \in A \cap u^{-1}[\{y\}]}{} \psi(x) \ud \HM^n y
    \end{equation}
    whenever $\psi : A \to \R$ is either non-negative and $\LM^n \restrict A$
    measurable or $\LM^n \restrict A$ measurable with $\psi \, |\J u|$
    $\LM^n \restrict A$ summable; in particular, the integrand on the
    right-hand side is $\HM^n$ measurable; see the proof
    of~\cite[Theorem~12]{Haj00} and~\cite[3.2.3(2)]{Fed69}.

    Define
    \begin{gather}
        Q(x) = \frac{|\uD u(x)|^{n-2} \uD u(x) \circ \uD u(x)^*}{|\J u(x)|}
        = \frac{|\uD u(x)|^n}{|\J u(x)|} \, Y(\uD u(x))
        \quad \text{for $x \in A$}
        \,,
        \\
        \text{and} \quad
        \theta_0(y) = \textsum{x \in A \cap u^{-1}[\{y\}]}{}
        \frac{|\uD \ u(x)|^n}{|\J u(x)|}
        \quad \text{for $y \in \R^{n+1}$} \,.
    \end{gather}
    Since $|A \circ A^*| \le |A|^2$ for linear maps~$A$, we have $|Q(x)| \,
    |\J u(x)| \le |\uD u(x)|^n$ for $x \in A$.

    Let $E \subseteq \R^{n+1} \without \cball 0{R_0}$ be a Borel set, and
    let $F \in \ccspace{\R^{n+1},\Hom(\R^{n+1},\R^{n+1})}$ with $\spt F \cap
    \cball 0{R_0} = \varnothing$. Then $F(u(x)) = 0$ whenever $|u(x)| \le
    R_0$. Since $\LM^n(P \without A) = 0$,
    \ref{mr:push:measure}\eqref{eq:stress:borel} and~\eqref{eq:structure:P} give
    \begin{equation}
        \label{eq:structure:A}
        \mu_{\rho}(E) = \textint{A \cap u^{-1}[E]}{} |\uD u|^n \ud \LM^n \,, \qquad
        \Sigma_\rho(F) = \textint{A}{} |\J u(x)| \, F(u(x)) \bullet
        Q(x) \ud \LM^n x \,.
    \end{equation}
    By~\ref{lem:hadamard}\eqref{eq:hadamard:amgm}, each term of the sum
    defining $\theta_0(y)$ is at least $n^{n/2}$. Hence $\theta_0(y) \ge
    n^{n/2} \card (A \cap u^{-1}[\{y\}])$ for $y \in \R^{n+1}$, and $\theta_0
    = 0$ outside $u[A]$. By~\eqref{eq:structure:area} with $\psi = |\uD
    u|^n |\J u|^{-1}$, the function $\theta_0$ is $\HM^n$ measurable. Hence
    $u[A] = \{ y : \theta_0(y) > 0 \}$ is $\HM^n$ measurable. By the same
    instance of~\eqref{eq:structure:area} and by~\eqref{eq:structure:A}
    with $E$ replaced by $\R^{n+1} \without \cball 0{R_0}$,
    \begin{displaymath}
        n^{n/2} \HM^n(u[A])
        \le \textint{}{}\theta_0 \ud \HM^n
        = \mu_{\rho} \bigl( \R^{n+1} \without \cball 0{R_0} \bigr) < \infty \,.
    \end{displaymath}
    Since $\HM^n$ is Borel regular~\cite[2.10.2(1)]{Fed69}, there is
    by~\cite[2.2.3]{Fed69} a Borel set $M \subseteq u[A]$ with $\HM^n(u[A]
    \without M) = 0$ and $\theta_0(y) < \infty$ for $y \in M$. If $y \in M$,
    the set $A \cap u^{-1}[\{y\}]$ is finite and non-empty, and we let
    \begin{displaymath}
        \Upsilon(y) = \textsum{x \in A \cap u^{-1}[\{y\}]}{} Q(x)
        = \textsum{x \in A \cap u^{-1}[\{y\}]}{} \frac{|\uD u(x)|^n}{|\J
        u(x)|} \, Y(\uD u(x)) \,.
    \end{displaymath}
    Since $A \subseteq P$, $Y(\uD u(x)) \in \Psi_\kappa$ for $x \in A$.
    Hence $\theta_0(y)^{-1} \Upsilon(y)$ is a convex combination of elements
    of $\Psi_\kappa$, and belongs to $\Psi_\kappa$, for $y \in M$.
    By~\eqref{eq:structure:area} applied to the coordinates of $Q$ with
    respect to a basis of $\Hom(\R^{n+1},\R^{n+1})$, the map $\Upsilon$ is
    $\HM^n \restrict M$ measurable. We apply~\eqref{eq:structure:area} to
    the function equal to $|\uD u|^n |\J u|^{-1}$ on $A \cap
    u^{-1}[E]$ and to $0$ on $A \without u^{-1}[E]$, and to the function
    mapping $x \in A$ onto $F(u(x)) \bullet Q(x)$. Since $u[A] \without M$ is
    an $\HM^n$ null set, \eqref{eq:structure:A} yields
    \begin{displaymath}
        \mu_{\rho}(E) = \textint{E \cap M}{} \theta_0 \ud \HM^n \,, \qquad
        \Sigma_\rho(F) = \textint{M}{} F \bullet \Upsilon \ud \HM^n \,.
    \end{displaymath}
    The function $\theta = \theta_0|M$ and the map $X = (\theta_0^{-1}
    \Upsilon)|M$ have the required properties.
\end{proof}

\section{Lower mass ratios and boundedness}
\label{sec:bounded}

\begin{miniremark}
    \label{mr:DGPR}
    In this section we combine~\ref{lem:structure} with the inequality of De~Philippis, Gennaioli,
    Pigati, and Rindler~\cite[Theorem~1.1]{DGPR26}. We obtain lower mass ratios, boundedness, and continuity. The
    standing hypotheses of~\ref{mr:stress} and the notation of section~\ref{sec:stress} are
    retained.

    We use~\cite[Theorem~1.1]{DGPR26} in the following form. Suppose $1 \le m \le d$ are integers and
    $\Psi \subseteq \Hom(\R^d,\R^d)$ is non-empty, compact and convex, with
    \begin{displaymath}
        \lambda = \inf \{ \sigma_m(X) : X \in \Psi \} > 0 \,, \qquad
        \Lambda = \sup \{ \sigma_1(X) : X \in \Psi \} \,.
    \end{displaymath}
    Then there is $0 < \Gamma < \infty$, depending only on $d$, $\lambda$ and $\Lambda$, with the
    following property. Suppose $M \subseteq \R^d$ is a Borel set with $\HM^m(M) < \infty$,
    $\vartheta : M \to [0, \infty)$ is $\HM^m \restrict M$ summable, $X : M \to \Psi$ is
    $\HM^m \restrict M$ measurable, and $\zeta : \ccspace{\R^d,\Hom(\R^d,\R^d)} \to \R$ is the
    linear functional given by
    \begin{displaymath}
        \zeta(F) = \textint{M}{} \vartheta \, F \bullet X \ud \HM^m
        \quad \text{for $F \in \ccspace{\R^d,\Hom(\R^d,\R^d)}$} \,.
    \end{displaymath}
    Suppose moreover that $\nu : \ccspace{\R^d,\R^d} \to \R$ is a linear
    functional with
    \begin{equation}
        \label{eq:DGPR:sup}
        \sup \bigl\{ \nu(g) : g \in \ccspace{\R^d,\R^d},\ |g(y)| \le 1
        \text{ for $y \in \R^d$} \bigr\} < \infty \,.
    \end{equation}
    If $f \in \ccspace{\R^d}^+$, $\sup f > 0$, $\xi \in \ccspace{\R^d,\R^d}$
    and $|\xi| \le f$, then $|(\sup f)^{-1} \xi| \le 1$. Hence $\nu$
    satisfies the hypothesis imposed on $T$ in~\ref{mr:push}, with
    $W = \R^d$, and $\|\nu\|(\R^d)$ equals the supremum
    in~\eqref{eq:DGPR:sup} by~\ref{mr:push}\eqref{eq:stress:variation}.
    Suppose finally that
    \begin{equation}
        \label{eq:DGPR:div}
        \zeta(\uD g) = - \nu(g)
        \quad \text{for $g \in \dspace{\R^d}{\R^d}$} \,.
    \end{equation}
    Then
    \begin{equation}
        \label{eq:DGPR}
        \textint{M}{} \vartheta \ud \HM^m \le \Gamma \, \HM^m(M)^{1/m} \, \|\nu\|(\R^d)
        \,.
    \end{equation}
    For $m = n$, $d = n+1$ and $\Psi = \Psi_\kappa$ with $n^{n/2} \le \kappa
    < \infty$, \ref{lem:K} gives $\lambda \ge (2 \kappa^2)^{-1}$ and $\Lambda
    \le 1$. We write $\Gamma_\kappa$ for the corresponding constant~$\Gamma$;
    it depends only on $n$ and~$\kappa$.
\end{miniremark}

\begin{miniremark}
    \label{mr:DGPR:language}
    In~\cite{DGPR26}, linear maps are represented by their coefficients with respect to the standard
    basis. In that language $\sigma_k$ is the $k$-th singular value and $\bullet$ is the
    Hilbert--Schmidt product.  Moreover, \cite[Theorem~1.1]{DGPR26} is formulated for vector valued
    measures.  To the functional $\zeta$ of~\ref{mr:DGPR} corresponds the measure with values in
    $\Hom(\R^d,\R^d)$ assigning $\textint{E \cap M}{} \vartheta X \ud \HM^m$ to each Borel set
    $E \subseteq \R^d$. To $\nu$ corresponds the $\R^d$ valued measure assigning
    $\textint{E}{} \eta(y)^*(1) \ud \|\nu\| y$ to $E$, where $\eta$ is the map
    of~\ref{mr:push}\eqref{eq:stress:polar} with $T = \nu$. Since
    $\eta(y)^*(1) \bullet w = \eta(y)(w)$ for $w \in \R^d$, \ref{mr:push}\eqref{eq:stress:polar}
    shows that integrating $g \in \ccspace{\R^d,\R^d}$ against this measure with respect to
    $\bullet$ gives $\nu(g)$. Since $|\eta(y)^*(1)| = \|\eta(y)\| = 1$ for $\|\nu\|$ almost all $y$,
    the total variation measure of this measure is~$\|\nu\|$.  Thus~\ref{mr:DGPR}\eqref{eq:DGPR:div}
    states that the second measure is the divergence of the first in the sense of distributions.
\end{miniremark}

\begin{miniremark}
    \label{mr:ode}
    In the proof of~\ref{lem:lower-mass}, \ref{lem:ode} is applied to $f(t)
    = \mu_\rho \, \oball at$ with $c = 2 n^{-1/2} \Gamma_\kappa$ and $b = h
    n^{-n/2} / 2$. There~\ref{lem:ode}\eqref{eq:ode:hyp} follows
    from~\ref{mr:DGPR}\eqref{eq:DGPR} applied with a cut-off function equal
    to~$1$ on $\cball at$ and with support in $\oball as$. The term $(s -
    t)^{-1} (f(s) - f(t))$ bounds the contribution of the gradient of the
    cut-off function, and the term $b f(s)$ bounds that of~$\tau_\rho$.
    Where $2 b c f^{1/n} \le 1$, the term $b f$ is absorbed, and the
    derivative of $f^{1/n}$ is at least $(2nc)^{-1}$. This is the argument
    of~\cite[5.1.6]{Fed69}, with~\ref{mr:DGPR} in place of the isoperimetric
    inequality~\cite[4.4.2(2)]{Fed69}. In~\cite[8.3]{All72}
    and~\cite[2.5]{Men09}, the first variation of a varifold takes the role
    of~$b f$, and its smallness corresponds to the condition $2 b c f^{1/n}
    \le 1$.
\end{miniremark}

\begin{lemma}
    \label{lem:ode}
    Suppose $0 < r < \infty$, $0 < c < \infty$, $0 \le b < \infty$, $f :
    (0, r] \to (0, \infty)$ is non-decreasing, and
    \begin{equation}
        \label{eq:ode:hyp}
        f(t) \le c \, f(s)^{1/n} \bigl( b f(s) + (s - t)^{-1} (f(s) - f(t))
        \bigr)
        \quad \text{whenever $0 < t < s \le r$} \,.
    \end{equation}
    Then $f(t) \ge \min \{ (2nc)^{-1} t, (2bc)^{-1} \}^n$ for $0 < t \le r$,
    where $(2bc)^{-1} = \infty$ if $b = 0$.
\end{lemma}

\begin{proof}
    Let $F = f^{1/n}$ and $0 < t \le r$. If $2 b c F(t) > 1$, then $f(t) >
    (2bc)^{-n}$. Suppose $2 b c F(t) \le 1$. We extend $F$ to a
    non-decreasing function on~$\R$ by setting $F(\tau) = 0$ for $\tau \le 0$
    and $F(\tau) = F(r)$ for $\tau > r$. By~\cite[2.9.19]{Fed69}, $F$ has a
    derivative $F'(\tau) \in \R$ at $\LM^1$ almost all $\tau \in \R$, and
    \begin{equation}
        \label{eq:ode:int}
        \textint{0}{t} F' \ud \LM^1 \le F(t) - F(0) = F(t) \,.
    \end{equation}
    Suppose $0 < \tau < t$ and $F'(\tau) \in \R$. Then $f$ is continuous
    at~$\tau$ and has the derivative $f'(\tau) = n F(\tau)^{n-1}
    F'(\tau)$. Letting $s \downarrow \tau$ in~\eqref{eq:ode:hyp} with $\tau$
    in place of~$t$, we obtain
    \begin{displaymath}
        f(\tau) \le b c F(\tau) f(\tau) + c F(\tau) f'(\tau)
        \le f(\tau) / 2 + n c f(\tau) F'(\tau) \,,
    \end{displaymath}
    where we used $F(\tau) \le F(t)$, $2 b c F(t) \le 1$ and $F(\tau)
    f'(\tau) = n f(\tau) F'(\tau)$. Since $0 < f(\tau) < \infty$, this gives
    \begin{equation}
        \label{eq:ode:diff}
        F'(\tau) \ge (2nc)^{-1}
        \quad \text{for $\LM^1$ almost all $\tau \in (0, t)$} \,.
    \end{equation}
    Combining~\eqref{eq:ode:int} and~\eqref{eq:ode:diff}, we obtain $F(t)
    \ge (2nc)^{-1} t$.
\end{proof}

\begin{miniremark}
    \label{mr:lower-mass:strategy}
    The proof of~\ref{lem:lower-mass} applies~\ref{mr:DGPR} to the linear
    functional mapping $F$ to $\Sigma_\rho(\varphi F)$, where $\varphi$ is a
    radial cut-off function centred at~$a$. The values of this functional on
    derivatives are computed from~\ref{mr:regular}\eqref{eq:regular}. The
    conclusion follows from~\ref{lem:ode}.
\end{miniremark}

\begin{lemma}
    \label{lem:lower-mass}
    Suppose $0 \le R_0 < \infty$, $n^{n/2} \le \kappa < \infty$,
    \ref{lem:structure}\eqref{eq:structure:hyp} holds, $\rho$ is regular as
    in~\ref{def:regular}, $a \in \spt \mu_{\rho}$, and
    $0 < r \le \min \{ \dist(a, K_\rho), |a| - R_0 \}$. Then
    \begin{displaymath}
        \mu_{\rho} \, \oball ar
        \ge (4 n^{1/2} \Gamma_\kappa)^{-n} \min \{ r, h^{-1} \}^n \,,
    \end{displaymath}
    where $h^{-1} = \infty$ if $h = 0$.
\end{lemma}

\begin{proof}
    Let $f(t) = \mu_{\rho} \, \oball at$ for $0 < t \le r$. Then $f$ is
    non-decreasing, and $0 < f(t) < \infty$ since $a \in \spt \mu_{\rho}$.
    Since $\dist(a, K_\rho) \ge r$ and $|y| > |a| - r \ge R_0$ for $y \in
    \oball ar$,
    \begin{equation}
        \label{eq:lower-mass:away}
        \oball ar \cap \bigl( K_\rho \cup \cball 0{R_0} \bigr) = \varnothing
        \,.
    \end{equation}
    Let $M$, $\theta$ and $X$ be as in~\ref{lem:structure}, and let $0 < t
    < s \le r$. Choose a function $\gamma : \R \to \R$ of class~$\infty$
    with $0 \le \gamma \le 1$, $\gamma(\tau) = 1$ for $\tau \le t$,
    $\gamma(\tau) = 0$ for $\tau \ge t + 3(s-t)/4$, and $0 \le - \gamma'
    \le 2 (s - t)^{-1}$. Define $\varphi(y) = \gamma(|y - a|)$ for $y \in
    \R^{n+1}$. Then $\varphi$ is of class~$\infty$, $0 \le \varphi \le 1$,
    $\spt \varphi \subseteq \oball as$, $\varphi(y) = 1$ for $|y - a| \le
    t$, and
    \begin{equation}
        \label{eq:lower-mass:grad}
        |\grad \varphi(y)| \le 2 (s - t)^{-1} \,, \qquad
        \grad \varphi(y) = 0 \text{ unless $t < |y - a| < s$} \,.
    \end{equation}
    Let $\zeta : \ccspace{\R^{n+1},\Hom(\R^{n+1},\R^{n+1})} \to \R$ and $\nu
    : \ccspace{\R^{n+1},\R^{n+1}} \to \R$ be the linear functionals given by
    \begin{displaymath}
        \zeta(F) = \Sigma_\rho(\varphi F) \,, \qquad
        \nu(g) = \textint{M}{} \theta \, g \bullet X(\grad \varphi) \ud \HM^n
        - \tau_\rho(\varphi g) \,.
    \end{displaymath}
    Since $\spt \varphi \subseteq \oball as$, \eqref{eq:lower-mass:away}
    and~\ref{lem:structure}\eqref{eq:structure} give
    \begin{equation}
        \label{eq:lower-mass:rep}
        \zeta(F) = \textint{M \cap \oball as}{} \varphi \theta \, F \bullet X \ud
        \HM^n
        \quad \text{for $F \in \ccspace{\R^{n+1},\Hom(\R^{n+1},\R^{n+1})}$}
        \,.
    \end{equation}
    Suppose $g \in \dspace{\R^{n+1}}{\R^{n+1}}$. Then $\varphi g \in
    \cnt{1}_c(\R^{n+1},\R^{n+1})$, $\spt(\varphi g) \subseteq \oball as
    \subseteq \R^{n+1} \without K_\rho$, and $\uD(\varphi g) = \varphi \uD g
    + \uD\varphi \cdot g$. From~\ref{mr:regular}\eqref{eq:regular},
    \ref{mr:notation}\eqref{eq:notation:rank-one},
    \eqref{eq:lower-mass:away} and~\ref{lem:structure}\eqref{eq:structure}
    we infer
    \begin{displaymath}
        \zeta(\uD g)
        = \Sigma_\rho(\uD(\varphi g)) - \Sigma_\rho(\uD\varphi \cdot g)
        = \tau_\rho(\varphi g)
        - \textint{M}{} \theta \, g \bullet X(\grad \varphi) \ud \HM^n
        = - \nu(g) \,.
    \end{displaymath}
    By~\ref{lem:K}, $\|X(y)\| = \sigma_1(X(y)) \le 1$ for $y \in M$. Suppose
    $g \in \ccspace{\R^{n+1},\R^{n+1}}$ and $|g(y)| \le 1$ for $y \in
    \R^{n+1}$. Then $|g \bullet X(\grad \varphi)| \le |\grad \varphi|$
    on~$M$ and $|\varphi g| \le \varphi$. Combining this
    with~\eqref{eq:lower-mass:grad}, \ref{lem:structure}\eqref{eq:structure}
    and~\ref{mr:push:measure}\eqref{eq:stress:bounds}, and recalling that $0 \le
    \varphi \le 1$ and $\spt \varphi \subseteq \oball as$, we obtain
    \begin{multline*}
        \nu(g)
        \le 2 (s - t)^{-1} \textint{M \cap \{ y : t < |y - a| < s \}}{} \theta \ud
        \HM^n + h n^{-n/2} \textint{}{}\varphi \ud \mu_{\rho}
        \\
        \le 2 (s - t)^{-1} \bigl( f(s) - f(t) \bigr) + h n^{-n/2} f(s) \,.
    \end{multline*}
    Hence $\nu$ satisfies~\ref{mr:DGPR}\eqref{eq:DGPR:sup}, and, by~\ref{mr:DGPR},
    $\|\nu\|(\R^{n+1}) \le 2 (s - t)^{-1} ( f(s) - f(t) ) + h n^{-n/2}
    f(s)$. By~\ref{lem:structure}\eqref{eq:structure},
    $\textint{M \cap \oball as}{} \theta \ud \HM^n = f(s) < \infty$. Since $\theta \ge n^{n/2}$ on $M$,
    we infer $\HM^n(M \cap \oball as) \le n^{-n/2} f(s)$. Since $\varphi(y) = 1$
    for $|y - a| \le t$, \ref{lem:structure}\eqref{eq:structure} also gives
    \begin{displaymath}
        \textint{M \cap \oball as}{} \varphi \theta \ud \HM^n
        \ge \textint{M \cap \oball at}{} \theta \ud \HM^n = f(t) \,.
    \end{displaymath}
    We apply~\ref{mr:DGPR} with $m = n$, $d = n+1$, $\Psi = \Psi_\kappa$, the
    Borel set $M \cap \oball as$, the function $\vartheta = \varphi \theta | M
    \cap \oball as$, the map $X | M \cap \oball as$, and the linear functionals
    $\zeta$ and~$\nu$. By~\eqref{eq:lower-mass:rep} and the preceding
    estimates, \ref{mr:DGPR}\eqref{eq:DGPR} yields
    \begin{displaymath}
        f(t) \le \Gamma_\kappa n^{-1/2} f(s)^{1/n} \bigl( 2 (s - t)^{-1}
        (f(s) - f(t)) + h n^{-n/2} f(s) \bigr) \,.
    \end{displaymath}
    This is~\ref{lem:ode}\eqref{eq:ode:hyp} with $c = 2 n^{-1/2}
    \Gamma_\kappa$ and $b = h n^{-n/2} / 2$. Since $2nc = 4 n^{1/2}
    \Gamma_\kappa$ and $(2bc)^{-1} = 2 n^{(n+2)/2} (4 n^{1/2} \Gamma_\kappa
    h)^{-1}$, \ref{lem:ode} with $t = r$ yields the conclusion.
\end{proof}

\begin{theorem}
    \label{thm:bounded}
    Suppose $0 \le h < \infty$, $H$ and $u$ are as in~\ref{def:hsys}, $0 \le
    R_0 < \infty$, $n^{n/2} \le \kappa < \infty$, and
    \begin{displaymath}
        |\uD u(x)|^n \le \kappa |\J u(x)|
        \quad \text{for $\LM^n$ almost all $x \in U$ with $|u(x)| > R_0$} \,.
    \end{displaymath}
    Then $u \in \Lp{\infty}(\oball 0\rho, \R^{n+1})$ for $0 < \rho < 1$.
\end{theorem}

\begin{proof}
    Let $0 < \rho < 1$. By~\ref{lem:sphere} there is a regular $\rho'$ with
    $\rho < \rho' < 1$. Let
    \begin{displaymath}
        c_0 = (4 n^{1/2} \Gamma_\kappa)^{-n} \min \{ 1, h^{-1} \}^n \,, \qquad
        W = \{ y : |y| \le R_0 + 1 \} \cup \{ y : \dist(y, K_{\rho'}) \le 1
        \} \,.
    \end{displaymath}
    Since $K_{\rho'}$ is compact, $W$ is bounded. If $a \in \spt \mu_{\rho'}
    \without W$, then $\dist(a, K_{\rho'}) > 1$ and $|a| - R_0 > 1$; hence
    $\mu_{\rho'} \, \oball a1 \ge c_0$ by~\ref{lem:lower-mass}. If $a_1,
    \ldots, a_N \in \spt \mu_{\rho'} \without W$ and $|a_i - a_j| \ge 2$ for
    $i \ne j$, then the balls $\oball{a_i}1$ are disjoint, and
    \begin{displaymath}
        N c_0 \le \mu_{\rho'}(\R^{n+1})
        = \textint{\oball 0{\rho'}}{} |\uD u|^n \ud \LM^n < \infty \,.
    \end{displaymath}
    Choosing such points with $N$ maximal, we obtain $\spt \mu_{\rho'}
    \without W \subseteq \bigcup_{i=1}^N \oball{a_i}2$. Hence there is $0 <
    R_1 < \infty$ with $\spt \mu_{\rho'} \subseteq \cball 0{R_1}$.

    Choose $\gamma : \R \to \R$ of class~$\infty$ with $\gamma(t) = 0$ for $t
    \le R_1$ and $\gamma(t) = 1$ for $t \ge R_1 + 1$, and define $\psi :
    \R^{n+1} \to \R^{n+1}$ by $\psi(y) = \gamma(|y|) \, y$ for $y \in
    \R^{n+1}$. Then $\psi$ is of class~$\infty$, $\uD\psi$ is bounded,
    $\uD\psi(y) = 0$ for $|y| \le R_1$, and $\psi(y) = y$ for $|y| \ge R_1 +
    1$. The proof of~\cite[Lemma~7.5]{GT01} applies to vector valued maps and
    shows that $w = \psi \circ u | \oball 0{\rho'}$ belongs to
    $W^{1,n}(\oball 0{\rho'},\R^{n+1})$ and $\uD w = \uD\psi(u) \circ \uD
    u$. Since $\mu_{\rho'}(\R^{n+1} \without \spt \mu_{\rho'}) = 0$
    by~\cite[2.2.5]{Fed69}, \ref{mr:push:measure}\eqref{eq:stress:borel} yields
    \begin{equation}
        \label{eq:bounded:alt}
        \text{$\uD u(x) = 0$ or $|u(x)| \le R_1$}
        \quad \text{for $\LM^n$ almost all $x \in \oball 0{\rho'}$} \,.
    \end{equation}
    Hence $\uD w = 0$ at $\LM^n$ almost all points of $\oball 0{\rho'}$.
    Since $\oball 0{\rho'}$ is convex, the Poincar{\'e}
    inequality~\cite[(7.45)]{GT01}, applied to the components of $w$, yields
    $c \in \R^{n+1}$ such that $w(x) = c$ for $\LM^n$ almost all $x \in
    \oball 0{\rho'}$. Since $u(x) = w(x)$ whenever $|u(x)| \ge R_1 + 1$, we
    obtain $|u(x)| \le \max \{ R_1 + 1, |c| \}$ for $\LM^n$ almost all $x \in
    \oball 0{\rho'}$. Hence $u \in \Lp{\infty}(\oball 0\rho, \R^{n+1})$.
\end{proof}

\begin{corollary}
    \label{cor:continuity}
    Suppose, in addition to the hypotheses of~\ref{thm:bounded}, that $H$ is
    continuous. Then $u$ has a representative which, for every compact subset
    $C$ of~$U$, is H{\"o}lder continuous on $C$ with some positive exponent.
    If moreover $\Lip (H|C) < \infty$ for every compact subset $C$
    of~$\R^{n+1}$, then there is $0 < \alpha < 1$ such that $u$ has a
    representative of class~$1$ whose derivative is H{\"o}lder continuous with
    exponent~$\alpha$ on every compact subset of~$U$.
\end{corollary}

\begin{proof}
    Each compact subset of~$U$ is contained in $\oball 0\rho$ for some $0 < \rho < 1$; hence, $u$ is
    locally bounded by~\ref{thm:bounded}. Recalling~\ref{def:hsys}\eqref{eq:hsys} one readily
    verifies that $u$ is a weak solution in the sense of~\cite[Definition~1.1]{FKLV18} of the
    system~\cite[(1.1)]{FKLV18} on the bounded Lipschitz domain~$U$, with $H$ replaced by
    $(-1)^n n^{-n/2} H$. This function is bounded and continuous. Under the additional hypothesis,
    its restriction to each compact subset of~$\R^{n+1}$ is Lipschitz.
    Hence~\cite[Proposition~1.1]{FKLV18} yields both conclusions.
\end{proof}

\begin{remark}
    \label{rem:direct}
    After the present work had been completed, F.~Rindler suggested to the author, in a private
    communication, that the method developed by Gennaioli and Rindler~\cite{GR26} for the Vanishing
    Mass Conjecture could be applied to $H$-systems directly, without passing through the
    Michael--Simon-type inequality of~\cite[Theorem~1.1]{DGPR26}. It would be interesting to carry
    this out. The approach presented here is, however, geometric. Through the push-forwards
    $\mu_\rho$ and $\Sigma_\rho$, it regards the image of~$u$ as a ``varifold'' whose first variation
    is controlled by the $H$-system. Boundedness then follows from lower mass ratios, as in the
    regularity theory of varifolds; see~\ref{mr:ode}. The author believes that this point of view is
    of independent interest, because it places the regularity of $H$-systems within geometric
    measure theory.
\end{remark}

\subsection*{Tool and computational resource disclosure}
\addcontentsline{toc}{section}{\numberline{}Tool and computational resource disclosure}

The author has had the idea to considering the image of~$u$ as a ``varifold'' in~$\R^{n+1}$ and
trying to prove lower density ratio bounds ever since reading the paper of Mou and Yang~\cite{MY96}
several years ago. This later became his motivation to study anisotropic geometric variational
problems. However, the main ingredient -- the monotonicity formula -- was not available.
When~\cite{DGPR26} (which builds on~\cite{GR26}) was released, the author immediately knew it could
be applied to $H$-systems. Claude Code (Opus 5 and 5.5) was given the author's pile of ten-year-old
scribbled \LaTeX{} notes, the~survey of Strzelecki and Schikorra~\cite[Problem~2.4]{SS17} for
context, the paper of De~Philippis, Gennaioli, Pigati, and Rindler~\cite{DGPR26}, and the main idea
of the author together with instructions on what the proof should look like. This produced the first
draft. After that, tens of subsequent prompts to Claude Code were employed to force a more
Federer-flavoured GMT presentation and also the use of results from the literature. The author
checked and adjusted the whole manuscript. He takes full responsibility for it.

\subsection*{Acknowledgements}
\addcontentsline{toc}{section}{\numberline{}Acknowledgements}

This~research was financed by the \href{https://ncn.gov.pl/}{National Science Centre Poland} grant
Sonata Bis number 2022/46/E/ST1/00328.

{\small
    \bibliography{hsys-michael-simon}{}
    \bibliographystyle{halpha}
}

\bigskip

{\small \noindent
  Sławomir Kolasiński
  \\
  Instytut Matematyki,
  Uniwersytet Warszawski
  \\
  ul. Banacha 2, 02-097 Warszawa, Poland
  \\
  \texttt{s.kolasinski@mimuw.edu.pl}
}

\end{document}